\documentclass[11pt]{amsart}

\usepackage{amsmath,amssymb,amsthm}
\usepackage{comment}
\usepackage[unicode,breaklinks=true,colorlinks=true]{hyperref}
\usepackage[dvipsnames]{xcolor}
\newcommand{\WW}{\textcolor{blue}}

\usepackage[top=1in, bottom=1in, left=1in, right=1in, marginparwidth=1in, marginparsep=0.1in]{geometry}

\numberwithin{equation}{section}
\newtheorem{theorem}{Theorem}[section]
\newtheorem{proposition}[theorem]{Proposition}
\newtheorem{lemma}[theorem]{Lemma}
\newtheorem{definition}[theorem]{Definition}

\theoremstyle{remark}
\newtheorem{remark}[theorem]{Remark}

\definecolor{darkblue}{rgb}{0,0,0.7}

\newcommand{\crr}[1]{{\color[rgb]{1,0,0} #1}} 

\newcommand{\crp}[1]{{\color[rgb]{0.5,0,0.8} #1}}

\newcommand{\al}{\alpha}

\newcommand{\de}{\delta}
\newcommand{\e}{\epsilon}

\newcommand{\la}{\lambda}

\newcommand{\si}{\sigma}
\newcommand{\td}{\tilde}

\newcommand{\R}{{\mathbb R }}
\newcommand{\N}{{\mathbb N}}
\newcommand{\Z}{{\mathbb Z}}

\newcommand{\nb}{{\nabla}}

\newcommand{\I}{\infty}

\newcommand{\donothing}[1]{{}}

\newcommand{\EQ}[1]{\begin{equation}\begin{split} #1 \end{split}\end{equation}}
\newcommand{\EQN}[1]{\begin{equation*}\begin{split} #1 \end{split}\end{equation*}}

\makeatletter
\newcommand{\xRightarrow}[2][]{\ext@arrow 0359\Rightarrowfill@{#1}{#2}}
\makeatother

\newcommand{\loc}{\mathrm{loc}}

\let\OLDthebibliography\thebibliography
\renewcommand\thebibliography[1]{
  \OLDthebibliography{#1}
  \setlength{\parskip}{1pt}
  \setlength{\itemsep}{1pt plus 0.3ex}
}

\title{Asymptotic stability for the 3D Navier-Stokes  {equations} in $L^3$ and nearby   spaces}
\author{Zachary Bradshaw and Weinan Wang}
\date{\today}

\begin{document}

\begin{abstract}
We provide a short proof of $L^3$-asymptotic stability around vector fields that are small in weak-$L^3$, including small Landau solutions. We show that asymptotic stability also holds for vector fields in the    range of Lorentz spaces strictly between $L^3$ and weak-$L^3$, as well as in the closure of the test functions in weak-$L^3$. To provide a comprehensive perspective on the matter, we observe that asymptotic stability of Landau solutions does not generally extend to weak-$L^3$ via a counterexample. 
\end{abstract}

\maketitle

\section{Introduction}

We consider the following perturbed version of the Navier-Stokes equations:
\begin{equation}\label{eq.PNS}
    \begin{cases}
        \partial_t u - \Delta u + u\cdot \nabla u +u\cdot \nabla U+U\cdot \nabla u+\nabla p=0\\
        \nabla \cdot u=0\\
        u(x,0)=u_0.
    \end{cases}
\end{equation}
where $U(x,t)\in L^\I(0,\I;L^{3,\I}  {(\mathbb R^3)})$ is divergence free with 
\[
\|U\|_{L^\I(0,\I;L^{3,\I} {(\mathbb R^3)})}\leq A<\I,
\]
for some $A$. We will also assume that $U\in C([0,\I); L^{3,\I} {(\mathbb R^3)})$. 
Note that $L^{3,\I}$ denotes the weak-$L^3$ space. It is the endpoint space in the nested scale of Lorentz spaces $L^{3,q}$ in which $L^{3,3} = L^3$.
The Navier-Stokes equations, which model the motion of viscous incompressible fluids, are obtained from \eqref{eq.PNS} by setting $U=0$.
If $U$ and $V$ are themselves solutions to the Navier-Stokes equations, possibly supplemented with a common forcing term,  then their difference $u=U-V$ solves \eqref{eq.PNS}. 
It is therefore the correct context to study asymptotic stability which asks:  
\begin{quote}If $U$ is a given solution to the stationary Navier-Stokes equations which is perturbed by $u_0$ to obtain a solution  $V$ to the evolutionary Navier-Stokes equations, does  the solution   $u=U-V$ to \eqref{eq.PNS}  go to zero in some sense as $t\to \I$? 
\end{quote}
This  problem has been studied in a number of contexts. 
If $U$ is a Landau solution---i.e.~a $-1$-homogeneous jet-entrained solution to the stationary Navier-Stokes equations satisfying an exact formula---then $L^2$-asymptotic stability was shown in \cite{KarchPil} provided the Landau solution is small. Note that in this application, while the Landau solution $U$ as well as the perturbed solution are forced, their difference, which solves \eqref{eq.PNS}, is not forced as the forces cancel.   This was later extended to general vector fields like $U$ \cite{karch}. 
The $L^3$-asymptotic stability of Landau solutions  was introduced by Li, Zhang and Zhang in \cite{LZZ} where it is shown that, if $u_0$ is small enough in $L^3$ and the background Landau solution is also sufficiently small, then there exists a unique global strong solution to \eqref{eq.PNS} for which $\|u\|_{L^3}\to 0$ as $t\to \I$. This has been generalized  in  \cite{ZZ,ZhZh}. 

We will provide a new perspective on the $L^3$-asymptotic stability in  \cite{LZZ}. The most visible difference in our work will be that $U$ is not restricted to the class of Landau solutions. Indeed, it can be any prescribed divergence free vector field satisfying the conditions below \eqref{eq.PNS} and does not need to satisfy any PDE. This relaxation is not merely academic as it will simplify the argument for asymptotic stability.  Relaxing the conditions on $U$ will necessitate a new treatment of the term $u\cdot \nabla U+U\cdot \nabla u$ because we cannot use Morrey's inequality as is done in \cite{LZZ}.  
A benefit of our approach is its flexibility which allows us to explore  asymptotic stability beyond $L^3$   by formulating our results for data in the Lorentz spaces $L^{3,q}$ where $3<q<\I$ and  data in the closure of the test functions under the $L^{3,\I}$ quasinorm. These spaces include progressively rougher data as evidenced by the chain of embeddings,
\[
L^3\subsetneq L^{3,3<q<\I} \subsetneq \overline{C_c^\I}^{L^{3,\I}}\subsetneq L^{3,\I}.
\] 
To round things out, we show that there exist initial perturbations $u_0$ of Landau solutions   in $L^{3,\I}$ which do not converge to the Landau solution in $L^{3,\I}$, regardless of how small the initial perturbation is.

Our first theorem concerns the well-posedness of \eqref{eq.PNS}.

\begin{theorem}[Global well-posedness]\label{thrm.globalWell}
    	 Let $u_0\in L^{3,q}$ with $3\leq q\leq \infty$ be divergence free. Let $U$ be given, also divergence free with $U\in C([0,\I); L^{3,\I})$\footnote{Note that inclusion in $C([0,\I);L^{3,\I})$ is understood to mean  strong continuity for $t>0$ and continuity in terms of $L^{3/2,1}$-$L^{3,\I}$ duality at $t=0$.} with 
      \[
      \sup_{0\leq t<\I} \| U\|_{L^{3,\I}}<A<\I.
      \]
      There exist $\e_1$ and $\e_2$ so that, if $A<\e_1$ and $\|u_0\|_{L^{3,q}}<\e_2$, then there exists a unique $u\in C([0,\I);L^{3,q})$ which solves \eqref{eq:MILDNS} and satisfies
		\[
		\|u \|_{L^\I(0,\I;L^{3,q})} 
  \leq C\| u_0 \|_{L^{3,q}},
		\]
  for a universal constant $C$.
\end{theorem}

We prove this using a modification of Kato's algorithm. To do this we first formulate a fixed point theorem tailored to the structure of \eqref{eq.PNS}. We then establish integral estimates  for the terms containing $U$ by splitting $U$ into a large-scale and small-scale part.  Ultimately, this leads to a mild solution of the form
\[u(x,t)=e^{t\Delta}u_0 - \int_{0}^{t}e^{(t-s)\Delta}\mathbb P\nabla\cdot (u \otimes u)\,ds
-
\int_{0}^{t}e^{(t-s)\Delta}\mathbb P(u\cdot \nabla U+U\cdot \nabla u)\,ds.\]
In essence, we are extending Kato's result and approach, which is for $U=0$, to a generalized version of the Navier-Stokes equations, \eqref{eq.PNS}, where $U\neq 0$ is small.  As will be visible in our proof, when $q<\I$ it suffices to have $U\in L^\I([0,\I);L^{3,\I})$; in particular, continuity is not needed. We include the continuity assumption as it allows us to outsource the proof of the $q=\I$ case to \cite{Meyer}.

When $q=3$ and $U$ is a small Landau solution, this result was proven in \cite{LZZ}. Let us briefly compare our approach to that of \cite{LZZ}. In \cite{LZZ}  the linear operator $\mathcal L v= -\Delta v +\mathbb P( v\cdot \nb U +U\cdot\nb v)$ is studied independently and a semigroup theory is developed for $e^{t\mathcal L}$.  Then, \eqref{eq.PNS} is formulated as an integral equation via the formula
\[
u(x,t)= e^{t\mathcal L}u_0 - \int_0^t e^{(t-s)\mathcal L}\mathbb P \nb \cdot (u\otimes u)\,ds.
\] 
This is essentially viewing the nonlinear  problem \eqref{eq.PNS} as a perturbation of $\partial_t v+\mathcal L v=0$.
Our approach avoids   the semigroup theory for $e^{t\mathcal L}$ by viewing \eqref{eq.PNS} as a perturbation of the heat equation.  


In \cite{LZZ}, data in either $L^p$ for some $3<p<\I$ or in $L^3$  with large norm are also considered and local well-posedness established. Since our primary interest is asymptotic stability, which is not meaningful for time-local solutions, we do not pursue these results but note they can be derived from our fixed-point theorem following Kato's argument.

\bigskip 

Our main asymptotic stability result is as follows.

\begin{theorem}[Asymptotic stability]\label{thrm.asymptoticStability}
    Suppose $3\leq q\leq \I$. For $U$,  $u_0$   and $u$ as in Theorem \ref{thrm.globalWell} but with $A\leq \e_1/2$ and $\|u_0\|_{L^{3,q}}\leq \e_2/2$ and with the extra assumption that $u_0\in \overline{C_{c,\si}^\I}^{L^{3,\I}}$ when $q=\I$, we have 
    \[
    \lim_{t\to \I}\|u \|_{L^{3,q}}(t)= 0.
    \]
\end{theorem}

Our proof of asymptotic stability  re-formulates $L^{3}$-asymptotic stability  in terms of $L^2$-asymptotic stability as studied by Karch et.~al.~\cite{karch}. In that paper, it is shown that
$L^2$-perturbations around vector fields like $U$, e.g., uniformly small in $L^{3,\I}$, are asymptotically stable.
Our observation is that, if we start with $U$ and perturb it by something small in $L^3$, call it $v$, then the perturbation $U+v$ can be written as $(U+V)+ u$ where $u$ is still small in $L^3$ but \textit{is also in an energy class} while $U+V$ is still small in $L^{3,\I}$---this has the form of the solutions for which $L^2$-asymptotic stability is proven in \cite{karch}.  It follows that $\|\nb u\|_{L^2}(t_k)\to 0$  for some sequence $t_k\to \I$. By a Sobolev embedding, we   have that the $L^6$ norm of $u$ is small at some time and, by interpolation, so is the $L^3$ norm. This means we can make the $L^3$ norm of $V+u$ as small as we like at a particular large time which depends on how small we want $V+u$ to be. Applying Theorem \ref{thrm.globalWell} at this time implies that the solution remains small at all later times. This leads to asymptotic stability.  A splitting argument also appears  in the proof of convergence in \cite{LZZ} (which is reminiscent of Calderon's \cite{Calderon}; see also \cite[p. 259]{LR}) but we note that our result is streamlined by the relaxation of Theorem \ref{thrm.globalWell} to velocities other than Landau solutions. In particular, when we split $U+v$ into $(U+V)+u$, we can use Theorem \ref{thrm.globalWell} to solve for $u$ instead of having to construct it by hand as in \cite{LZZ}.

This argument can be extended to the Lorentz spaces $L^{3,q}$ when $q<\I$ because the closure of $C_c^\I$ under the $L^{3,q}$ norm is all of $L^{3,q}$. By definition, this property also holds in $\overline{C_c^\I}^{L^{3,\I}}$. This however fails in general when $q=\I$ meaning that we cannot decompose $v$ into $V+u$ as in the above picture.  As justified in the following theorem, this failure cannot be avoided.

\begin{theorem}[Asymptotic stability fails in $L^{3,\I}$]\label{thrm.excluded}
    Let $U = u_L(x)$ be a  Landau solution which satisfies the size requirement in Theorem \ref{thrm.globalWell}, i.e.~$\|u_L\|_{L^{3,\I}} <\e_1$. For any $\e\in (0,\e_2)$, there exists $u_0\in L^{3,\I}$ for which $\|u_0\|_{L^{3,\I}}< \e$ so that  
    \[
    \limsup_{t\to \I }\|u \|_{L^{3,\I}} >0,
    \] 
    where $u$ is the solution to \eqref{eq.PNS} referenced in Theorem \ref{thrm.globalWell}.
\end{theorem}

 In other words, asymptotic stability around Landau solutions fails for some initial perturbations in $L^{3,\I}$ regardless of how small the Landau solution or the initial perturbation are in $L^{3,\I}$ and, therefore, Theorem \ref{thrm.asymptoticStability} cannot be generalized to  $L^{3,\I}$. Of course, Theorem \ref{thrm.globalWell} implies the perturbed solution is  stable in that it remains within a finite distance in $L^{3,\I}$ of the Landau solution, provided the initial difference is small. The initial perturbations we use in the theorem are scaling invariant. Classically, for the Navier-Stokes equations, if small-data global well-posedness holds in a class admitting self-similar initial data, e.g.~in $L^{3,\I}$, then the global solution associated with a sufficiently small self-similar initial datum is itself self-similar. Since the $L^{3,\I}$ norm of a self-similar solution is independent of time, it cannot go to zero. Because Landau solutions are self-similar, the same argument applies to the \textit{perturbed} Navier-Stokes equations \eqref{eq.PNS} around a Landau solution $U$.

\bigskip 
\noindent \textbf{Organization:}   Section 2 contains definitions and preliminary ideas. Section 3 contains the fixed point argument and the proof of Theorem \ref{thrm.globalWell}. Theorems \ref{thrm.asymptoticStability} and \ref{thrm.excluded} are proven in Section \ref{sec:AS}.

\section{Definitions and preliminaries}

First, we define  Lorentz spaces. 
		\begin{definition}\label{lorentz}
			For a measurable function $f: \Omega \rightarrow \mathbb R$, we define:
		\begin{equation*}
			d_{f,\Omega}(\alpha):=
			|\{x\in \Omega: |f(x)|>\alpha\}|
			.
		\end{equation*}
			Then, the Lorentz spaces $L^{p,q}(\Omega)$ with $1\leq p<\infty$, $1\leq q \leq \infty$ is the set of all functions $f$ on $\Omega$ such that the quasi-norm $\|f\|_{L^{p,q}(\Omega)}$ is finite and
			\begin{equation*}
			\|f\|_{L^{p,q}(\Omega)}
			:=
			\left(p\int_{0}^{\infty}\alpha^{q} d_{f,\Omega}(\alpha)^{\frac{q}{p}}\frac{d\alpha}{\alpha}\right)^{1/q}
			\end{equation*}
   \begin{equation*}
			\|f\|_{L^{p,\infty}(\Omega)}
			:=
			\sup_{\alpha>0} \alpha d_{f,\Omega}(\alpha)^{1/p}
			.
			\end{equation*}
		\end{definition}
		The space $L^{p, \infty}$ coincides with  weak-$L^p$. We also have    $\|f\|_{L^{p,p}(\Omega)}=\|f\|_{L^{p}(\Omega)}$ and $L^{p,q_1}(\Omega) \subset L^{p, q_2} (\Omega)$ whenever $1 \leq q_1 \leq q_2 \leq \infty$, with the embedding being continuous.

The following is a standard heat semigroup estimate \cite[Proposition 3.2]{HN}.
  \begin{proposition}[Heat estimate]\label{heat}
      Let $1<p\leq r<\infty$ and $1< q\leq \infty$, then
      \[
      \|e^{t\Delta}f\|_{L^{p_1, q}(\R^n)}
		\lesssim
		t^{ -\frac{n}{2}(\frac{1}{p_2}-\frac{1}{p_1})}
		\|f\|_{L^{p_2, q}(\R^n)}.
      \]
  \end{proposition}

The next lemma is Young's convolution inequality in Lorentz spaces. It is also known as ``O'Neil's convolution inequality'' and a variation on what originally appeared as  \cite[Theorem 2.6]{o1963convolution}.  {We use the version in Blozinski \cite[Theorem 2.12]{Blo} which characterizes the constants more precisely than in \cite{o1963convolution}.  
 
\begin{lemma}[Young's convolution inequality in Lorentz spaces, \cite{Blo}]
\label{Lemma_young}
		Suppose $f \in L^{p_1, q_1} (\R^3)$, $g \in L^{p_2, q_2} (\R^3)$ with $1 < p_1, p_2, r < \infty$ and $0 < q_1, q_2, s \leq \infty$,
		\begin{equation*}
	1/r+1=1/p_{1}+1/p_{2}
			\quad
			\text{and}
			\quad
			1/s\leq 1/q_{1}+1/q_{2}
		\end{equation*}
		Then $f *g \in L^{r,s}(\R^3)$ and 
		\begin{equation*}
		\|f*g\|_{L^{r,s}(\R^3)}
			\leq
			 C(r, q_1, q_2, s) \|f\|_{L^{p_{1},q_{1}}(\R^3)}
			\|g\|_{L^{p_{2},q_{2}}(\R^3)},
		\end{equation*}
  where
  \begin{equation}\label{e.w09161}
      C(r, q_1, q_2, s)
      \begin{cases}
          = O (r(\alpha^{\frac{1}{\alpha}-\frac{1}{s}})),~\text{if }1/\alpha=1/q_1 + 1/q_2, s\geq 1
          \\
          \leq 
          O(2^{s/r}-1)^{-1/s}(\alpha^{\frac{1}{\alpha}-\frac{1}{s}}),~\text{if }
          1/\alpha=1/q_1 + 1/q_2, 0<s<1.
      \end{cases}
  \end{equation}
	\end{lemma}
 
\begin{remark}\label{Remark.oseen}
    By the preceding inequality it is easy to see that, letting $K$ be the kernel of the Oseen tensor we have  for $3<p<\I$ and $C=C(p, \infty)=O(p(p^{\frac{1}{p}-\frac{1}{p}}))=O(p)$   (in other words, we use the first case of \eqref{e.w09161} with $r=s=q_1=\alpha=p \geq 1$),
\[\label{e.w9162}
\| D^\al \mathbb P e^{t \Delta}  f \|_{L^p} \leq C {p} \| D^\al K(\cdot ,t) \|_{L^{3p/(3+2p),p}} \|f \|_{L^{3,\I }}\leq C {p} \| D^\al K(\cdot ,t) \|_{L^{3p/(3+2p)}} \|f \|_{L^{3,q }},
\]
where  $\al$ is a multi-index in $\N_0^n$ and we have used the embeddings $L^{r,s} \subset L^{r,s'}$ for $s'>s$ twice, noting that $3p/(3+2p)<p$ for all $p$.
It follows that 
\[
\| D^\al \mathbb P e^{t \Delta}  f \|_{L^p} \leq  C   {p} t^{-|\al|/2 - 3(1/3-1/p)/2}\|f \|_{L^{3,q}},
\]
where the constant depends on $|\al|$.   
\end{remark}

\subsection{The weak solutions of Karch et. al. \cite{karch}}

In a series of papers \cite{KarchPil,karch}, Karch and Pilarczyk, along with Schonbek in \cite{karch},   establish asymptotic stability for a class of weak solutions generalizing the Leray-Hopf weak solutions for Navier-Stokes to the perturbed Navier-Stokes equations.  We recall the following definition from \cite{karch}.
\begin{definition}Let $u_0\in L^2$ and $T\in (0,\I]$.
    A vector field $u$ is a weak solution to \eqref{eq.PNS} on $\R^3\times [0,T]$ if it  satisfies \eqref{eq.PNS} in a weak sense (see \cite[Def.~2.6]{karch} for the precise definition of this) and belongs to the space
    \[C_w^\I([0,T];L^2_\si (\R^3 ) )\cap L^2((0,T]; \dot H^1(\R^3)),\]
    where $L^2_\si$ is the closure  of divergence free test functions in $L^2$. 
\end{definition}

We will use the following theorem of Karch et.~al., which is \cite[Theorem 2.7]{karch}.

\begin{theorem}[$L^2$-asymptotic stability]\label{thrm.karch}For every $u_0\in L^2_\si$, $U$ as given below \eqref{eq.PNS} and each $T>0$, the problem \eqref{eq.PNS} has a weak solution $u$ for which the strong energy inequality 
\[
\| u(t)\|_{L^2}^2 + 2 (1- A K ) \int_s^t \|\nb u \|_{L^2}^2\,ds\leq \| u (s)\|_{L^2}^2,
\]
holds for almost every $s\geq 0$ (including $s=0$) and every $t\geq s$,
where $K$ is a universal constant and we are assuming $1-AK>0$ (this amounts to a smallness condition on $A$).  Furthermore we have 
\[
\lim_{t\to \I} \| u(t)\|_{L^2}^2 = 0 .
\]
\end{theorem}

We will also need a weak-strong uniqueness result which connects the solutions we construct in Theorem \ref{thrm.globalWell} to those in Theorem \ref{thrm.karch}.

\begin{theorem}[Weak-strong uniqueness]\label{thrm.weakstronguniqueness}
    Suppose $u_0\in L^{3,q}\cap L^2_\si$ and is small in $L^{3,q}$ as required by Theorem \ref{thrm.globalWell}. Let $u$ denote the global weak solution in Theorem \ref{thrm.karch}. Let $v$ denote the global strong solution in Theorem \ref{thrm.globalWell}. Then $u=v$. 
\end{theorem}
\begin{proof}[Proof sketch] The details of this sort of proof are well known when $q=3$---see, e.g., \cite[Theorem 4.4]{Tsai-book}.  The only modification here is the use of the estimate\footnote{Note that this is the genesis of the constant $K$ appearing in Theorem \ref{thrm.karch}.} 
\[
\int  f\cdot \nb U g \,dx    \leq K \| U \|_{L^{3,\I}} \| \nb f\|_{L^2} \| \nb g\|_{L^2}.
\]
In the context of a typical weak-strong uniqueness proof, this shows up when bounding 
\[
\int w\cdot\nb v w\,dx \leq K \| v\|_{L^{3,\I}} \| \nb w\|_{L^2}^2,
\] 
where $w=u-v$. 
By taking $ \| v\|_{L^{3,\I}}\lesssim K^{-1}$, which amounts to a smallness condition in Theorem \ref{thrm.globalWell}, formal energy estimates can be closed.  When $3\leq q<\I$, this argument still applies because $L^{3,q}$ embeds continuously in $L^{3,\I}$.
\end{proof}

\section{A fixed point argument} 

Recall the following fixed point theorem: If $E$ is a Banach space and $B: E\times E \to E$ is a bounded bilinear transform satisfying \EQ{
\| B(e,f)\|_{E}\leq C_B \| e\|_E\|f\|_E,
}
and if $\|e_0\|_{E}\le \varepsilon \le (4C_B)^{-1}$, then the equation $e = e_0 - B(e,e)$ has a solution with $\| e\|_{E}\leq 2 \varepsilon$ and this solution is unique in $\overline B(0,2\varepsilon)$. We make use of the following linear perturbation of this.

\begin{proposition}\label{prop.fp}
If $E$ is a Banach space and $B: E\times E \to E$ is a bounded bilinear transform satisfying \EQ{\label{contraction}
\| B(e,f)\|_{E}\leq C_B \| e\|_E\|f\|_E,
}
and if $\|e_0\|_{E}\le \varepsilon\le (4C_B)^{-1}$, and  $U$ is given and satisfies,
\EQ{\label{contraction2}
\| B(e,U)\|_{E} + \| B(U,e)\|_{E}\leq \frac 1 {8} \|e\|_E,
}
then the equation $e = e_0 - B(e,e)-B(U,e)-B(e,U)$  has a solution with $\| e\|_{E}\leq 3 \varepsilon/2$ and this solution is unique in $\overline B(0,3\varepsilon/2)$.
\end{proposition}

\begin{proof}
One just sets up a Picard scheme with 
\[
e_n = e_0 -B(e_{n-1},e_{n-1})-B(U,e_{n-1}) -B(e_{n-1},U).
\]
We have 
\[
\|e_1\|_E \leq \varepsilon + \frac 1 4 \varepsilon    + \frac 1 8 \varepsilon \leq \frac 3 2 \varepsilon.
\] 
In principle, $C_B$ is large and, in particular, we assume $C_B\geq 1$. By induction, if $e_{n-1}$ satisfies the bound written above for $e_1$ then 
\EQ{
\|e_n\|_E &\leq \frac 3 2 \varepsilon 
}
and, therefore, for all $n\in \N$,
\[
\|e_n\|_E\leq \frac 3 {2 } \varepsilon.
\]
We have also that 
\[
e_{n+1}-e_{n}=-B(e_{n},e_n) -B(U,e_n)-B(e_n,U) + B(e_{n-1},e_{n-1})+B(U,e_{n-1})+B(e_{n-1},U).
\]
It follows that 
\EQ{
\|e_{n+1}-e_n\|_E
&\leq C_B\| e_{n+1}\|_E\|e_{n}-e_{n-1}\|_E +C_B\| e_{n}\|_E\|e_{n}-e_{n-1}\|_E
        +\frac 1 {4} \|e_{n}-e_{n-1}\|_E
\\&\leq \left( \frac 3 4 +\frac 1 8 \right)\|e_{n}-e_{n-1}\|_E.
}
This implies the sequence is Cauchy and therefore has a limit  $e$ in $E$ and  $\|e\|_E\leq 3\varepsilon/2$. The uniform bounds and continuity of the bilinear operator guarantee that $e = e_0 - B(e,e)-B(U,e)-B(e,U)$. If $f$  satisfies $ f= e_0 - B(f,f)-B(U,f)-B(f,U)$ and $\|f\|_E\leq 7/(16C_B)$, then 
\[
\|e-f\|_E\leq \frac {7}{8} \|e-f\|_E,
\]
implying $e$ is unique.
\end{proof}

In what follows we will apply this to the operator
\[
B (u,v)= -\int_0^t e^{(t-s)\Delta}\mathbb P \nb \cdot (u\otimes v)\,ds.
\]
where $\mathbb P$ is the Leray projector. 
 Applying the Leray projection to \eqref{eq.PNS} as well as Duhamel's principle results in the following mild formulation of \eqref{eq.PNS},
\EQ{\label{eq:MILDNS}
u(x,t)&=e^{t\Delta}u_0 - \int_{0}^{t}e^{(t-s)\Delta}\mathbb P\nabla\cdot (u \otimes u)\,ds
-
\int_{0}^{t}e^{(t-s)\Delta}\mathbb P(u\cdot \nabla U+U\cdot \nabla u)\,ds
\\&=
e^{t\Delta}u_0+B(u,u)+B(u,U)
+B(U,u).
}

 We are now ready to prove our global well-posedness result. 
	\begin{proof}[Proof of Theorem \ref{thrm.globalWell}]
  We first prove the case when $3\leq q<\I$. The case $q=\I$ will be given at the end of the proof.
		We  define a Kato-type space: 
		\[
		\| \cdot \|_K = \sup_{3< p<\I} \|\cdot \|_{K_p},
		\]
		where 
		\[
		\| u \|_{K_p}=\sup_{0<t<\I}  {\frac 1 p} t^{1/2-3/(2p)} \| u\|_{L^p}(t).
		\]
            { The appearance of $p^{-1}$ reflects the appearance of $p$ in the constants on the right-hand side of the final display in Remark \ref{Remark.oseen}. Note that to get estimates for \eqref{eq.PNS} with $U=0$ it suffices to only consider several of the Kato spaces $K_p$. It seems to treat the generalized case they must all be included as we will eventually need to estimate $\|u\|_{K_p}$ in terms of $\|u\|_{K_{2p}}$.} 
		
  Let $\| \cdot \|_X = \sup\{  \|\cdot\|_K ,\|\cdot\|_Y\}$, where $Y=L^\I((0,\I);L^{3,q})$.  Our strategy is to apply the fixed point theorem with $E=X$.

  \subsection{Bilinear estimates} 
For a value $\de>0$ which we will eventually specify, we write $ U(x,t) = U_{low}+ U_{high}$ where
		\[
		U_{high}= U \chi_{\{|U|\geq  \delta \sqrt t^{-1}  \}}.
		\]
		Let 
		\[
		S_t=\{|U|\geq   \delta \sqrt t^{-1} \}.
		\]
		Note that $\| U_{high} \|_{L^{3,\I}}\leq \| 
		U\|_{L^{3,\I}} $   and $|S_t| \leq \bigg( \frac {\sqrt t} {\delta} \bigg)^3 \| U_{high}\|_{L^{3,\I}}^3$.
		On the other hand, $\| U_{low}\|_{L^\I} (t) \leq   \delta \sqrt t^{-1} $.  
		Using this and Remark \ref{Remark.oseen} we see that
		\EQN{
			\| B(u,U_{low})\|_{L^p}(t)&\lesssim  {p} \delta \int_0^t \frac 1 {(t-s)^{1/2+ 3(1/3-1/p)/2} s^{1/2} }\|u\|_{L^{3,q}}(s)\,ds
			\\&\lesssim    {p}  \delta \int_0^t \frac 1 {(t-s)^{1-3/(p2)} s^{1/2} }\|u\|_{L^{3,q}}(s)\,ds
			\\&\lesssim  {p}  \delta t^{3/(2p)-1/2} \|u\|_{Y}.
		}
		Hence
		\EQ{\label{BUlowSub}
			\sup_{0<t<\I}t^{1/2-3/(2p)} {\frac 1 p}\| B(u,U_{low})\|_{L^p}(t) 
			&\lesssim \delta  \|u\|_{X},
		}
  which means
  \[\label{e.w07102}
  \|B(u,U_{low})\|_{K}
  \lesssim  \delta \|u\|_{X},
  \]
To estimate $\|B(u, U_{low})\|_{Y}$, we have  
  \EQ{\label{BUlowCrit}
		\| B(u,U_{low})\|_{L^{3,q}} \lesssim \delta\int_0^t \frac 1 {(t-s)^{1/2}} \| uU_{low}\|_{L^{3,q}}\,ds
  \lesssim \delta \int_0^t \frac 1 {(t-s)^{1}s^{1/2}} \| u\|_{L^{3,q}}\,ds.
  }
Thus, we see 
  \[\label{e.w07103}
		\| B(u,U_{low})\|_{Y} \lesssim \delta\|u\|_X.
		\]

We now turn our attention to the singular part of $U$.
Observe that  for $q\geq 3$, 
\EQ{
\| B(u,U_{high})\|_{L^{3,q}}  \lesssim  \| B(u,U_{high})\|_{L^{3}} &\lesssim \int_0^t \frac 1 {(t-s)^{3/4}} \| u\otimes U_{high}\|_{L^2} (s)\,ds
\\&\lesssim  {\frac 1 {10}} \int_0^t \frac 1 {(t-s)^{3/4}} \|u\|_{L^{10}}(s) \| U_{high}\|_{L^{5/2}}(s)\,ds.
}
Note that by \cite[Lemma 6.1]{Tsai2020},
\[
 \| U_{high}\|_{L^{5/2}}^{5/2}(s) \lesssim (\sqrt s)^{1/2}\delta^{-1/2} \| U\|_{L^{3,\I}}^{5/2} .
\]
Hence, 
\EQ{\label{BUhighCrit}
\| B(u,U_{high})\|_{L^{3,q}} &\lesssim \int_0^t \frac {\|u\|_{\mathcal K_{10}}} {(t-s)^{3/4} s^{1/2 - 3/(20)}   }  ((\sqrt s)^{1/2}\delta^{-1/2} \| U\|_{L^{3,\I}}^{5/2} )^{2/5}\,ds
\\& \lesssim \delta^{-1/5}\|U\|_{L^\I_tL^{3,\I}_x}  \|u\|_{ K_{10}} \int_0^t \frac 1 {(t-s)^{3/4}s^{1/4}}\,ds
\\&\lesssim \delta^{-1/5}\|U\|_{L^\I_tL^{3,\I}_x}  \|u\|_{X}.
}
Regarding estimates in $K$, observe that for $3<p<\I$,\footnote{Compared to Kato's original paper \cite{Katao1984}, we need to include the full range of $K_p$ norms $3<p<\I$ as we are only able to bound $K_p$ using $K_{2p}$. }
\EQ{
 {\frac 1 p} \| B(u,U_{high})\|_{L^p}(t)&\lesssim  {\frac 1 p}\int_0^t \frac 1 {(t-s)^{1/2 + 3(1/r-1/p)/2}}\|u\otimes U_{high}\|_{L^r}(s)\,ds
\\&\lesssim  {\frac 1 {2p}} \int_0^t \frac 1 {(t-s)^{1/2 + 3(1/r-1/p)/2}}\| U_{high}\|_{L^{\bar r}}(s)\|u\|_{L^{2p}}\,ds
}
where we will need $r<3$, $3(1/r-1/p)/2<1/2$ and  \[
\frac 1 r = \frac 1 {\bar r} + \frac 1 {2p}.
\]
Provided $\bar r<3$ we have 
		\[
		\int_{S_s} |U_{high}|^{\bar r}\lesssim  \|U_{high}\|_{L^{3,\I}}^{\bar r}|S_s|^{1-{\bar r}/3} = \|U_{high}\|_{L^{3,\I}}^{\bar r}\big(  {\sqrt s} / \delta\big)^{3-{\bar r}}.
		\]
Hence,
\EQ{
& {\frac 1 p}\| B(u,U_{high})\|_{L^p}(t)\lesssim {\frac 1 {2p}} \int_0^t \frac 1 {(t-s)^{1/2 + 3(1/r-1/p)/2}} \|U_{high}\|_{L^{3,\I}}^{}\big(  {\sqrt s} / \delta\big)^{3/\bar r-{1}}  \|u\|_{L^{2p}}(s)\,ds 
\\&\lesssim \delta^{1-3/\bar r}   \|u\|_{K_{2p}}  \|U \|_{L^\I_t L^{3,\I}_x}\int_0^t \frac {s^{-1/2 + 3(1/r-1/(2p))/2}} {(t-s)^{1/2 + 3(1/r-1/p)/2}s^{1/2-3/(4p)}}     \,ds
\\&\lesssim \delta^{1-3/\bar r}   \|u\|_{X}  \|U \|_{L^\I_t L^{3,\I}_x} t^{3/(2p)-1/2},
}
implying
\EQ{\label{BUhighSub}
\| B(u,U_{high})\|_{K_p}\lesssim \delta^{1-3/\bar r}   \|u\|_{X}  \|U \|_{L^\I_t L^{3,\I}_x}.
}
For the preceding argument to make sense we needed to have 
\[
r<3; \quad 3(1/r-1/p)<1/2; \quad \bar r<3. 
\]
The middle condition and last condition are met provided
\[
\frac {3p} {p+3}< r < \frac {6p} {2p+3}.
\]
As  $\frac {3p} {p+3}<\frac {6p} {2p+3}<3$ for all $3<p<\I$, we can always choose an appropriate $r\in (3/2,3)$. At this stage we have confirmed that
\[
\| B(u,U_{high})\|_{X} \lesssim  (\delta^{1-3/\bar r} + \delta^{-1/5})   \|u\|_{X}  \|U \|_{L^\I_t L^{3,\I}_x}.
\]
Note that since $\bar r\in (3/2,3)$, assuming $\de<1$, the dependence on $\bar r$ can be eliminated above and we obtain 
\[
\| B(u,U_{high})\|_{X} \lesssim  \delta^{-1} \|u\|_{X}  \|U \|_{L^\I_t L^{3,\I}_x}.
\]
So, by first taking $\delta$ small and basing a smallness condition on $ \|U \|_{L^\I L^{3,\I}}$ in terms of $\delta$ and universal constants, we can ensure that 
\[
\| B(u,U)\|_{X} \leq \frac 1 {16} \| u\|_X.
\]
By a symmetric argument we have the same bound for $\| B(U,u)\|_{X} $.

We have explicitly worked out the bilinear estimates for the terms involving $U$ but have not mentioned $B(u,u)$.  {Inspecting the estimates above, we may replace $B(u,U)$ with $B(u,u)$ and set $\delta=1$ to obtain
\[
\| B(u,u)\|_{X} \lesssim \| u\|_X \|u\|_{L^\I_t L^{3,\I}_x}\lesssim \| u\|_{X}^2,
\]
where we have used the continuous embedding $L^{3,q}\subset L^{3,\I}$.}
The suppressed constant in the preceding estimate becomes $C_B$ in Proposition \ref{prop.fp}. At this stage, we have confirmed that by  requiring  $A \geq  \|U\|_{L^\I_tL^{3,\I}_x}$ to be small, we can apply Proposition \ref{prop.fp} to obtain a unique solution $u$ to \eqref{eq.PNS} which is in $X$.

\subsection{Time continuity.} We now show continuity in time of the  $L^{3,q}$ norms of $u$ for $t>0$. Continuity at $t=0$---i.e.~convergence to the initial data---will be addressed after this. We follow the approach in \cite{Tsai-book}, which is based on \cite{Katao1984}.

We begin by establishing continuity of the caloric extension of $u_0$. Taking $t,t_1>0$,   we want to control
   $e^{t\Delta}u_0-e^{t_1\Delta}u_0$ in $L^{3,q}$. We have as $t\rightarrow 0$
  $$
\left(e^{t \Delta} f-f\right)(x)=\int_{\mathbb{R}^d} e^{-|z|^2 / 4} g(x, z, t) d z, \quad g(x, z, t)=f(x-\sqrt{t} z)-f(x).
$$
By Proposition \ref{heat} and Minkowski's integral inequality in \cite{Mandel},
$$
\left\|e^{t \Delta} f-f\right\|_{L^{p,q}} 
\leq 
\int_{\mathbb{R}^d} e^{-|z|^2 / 4}\|g(\cdot, z, t)\|_{L^{p,q}} dz \rightarrow 0,
$$
 where we emphasize that $p,q<\I$---this fails when $p<\I$ and $q=\I$.
Next, we see by Young's convolution inequality in Lemma \ref{Lemma_young},
\begin{equation}\label{e.w08231}
\begin{split}
    \left\|e^{(t+h) \Delta} f-e^{t \Delta}f\right\|_{L^{3,q}}
    &\leq
    \|(4 \pi (t+h))^{-3 / 2} e^{-x^2 / 4 (t+h)}-(4 \pi t)^{-3 / 2} e^{-x^2 / 4 t}
    \|_{L^{1}}
    \|f\|_{L^{3,q}}
    \rightarrow 0,
\end{split}
\end{equation}
as $h\rightarrow 0$.


  For the time continuity of the Duhamel terms, we have 
  \begin{equation}
  \begin{split}
  B(u,U)(t)-B(u,U)(t_1)
  &=
  \int_0^t e^{(t-s)\Delta}\mathbb P \nb \cdot (u\otimes U)\,ds
  -
  \int_0^{t_{1}} e^{(t_1-s)\Delta}\mathbb P \nb \cdot (u\otimes U)\,ds
  \\&=
  \int_{0}^{\rho t_1}
  \left(e^{(t-\rho t_1)\Delta}-e^{(t_1-\rho t_1)\Delta}\right)e^{(\rho t_1-s)\Delta}\mathbb P \nb \cdot (u\otimes U)
  \,ds
  \\&\quad+
  \int_{\rho t_1}^t e^{(t-s)\Delta}\mathbb P \nb \cdot (u\otimes U)\,ds
    -
  \int_{\rho t_1}^{t_{1}} e^{(t_1-s)\Delta}\mathbb P \nb \cdot (u\otimes U)\,ds
    \\&=
    I_1 + I_2 +I_3,
  \end{split}
  \end{equation}
  where we take $\rho\in (0,1)$ so that $\rho t_1 <t$ and let $t_1$ be fixed, and will let $t$ approach $t_1$ from either side. Introducing $\rho$ allows us to prove left and right continuity simultaneously.

  To estimate $I_1$, we again use our decomposition of $U=U_{high}+U_{low}$, this time taking $\delta=1$.
  For $I_1$, we have  
  \begin{equation}
  \begin{split}
  \|I_1\|_{L^{3,q}}
  &\leq
  \int^{\rho t_1}_{0} 
  \left\|\left(e^{(t-\rho t_1)\Delta}-e^{(t_1-\rho t_1)\Delta}\right)e^{(\rho t_1-s)\Delta}\mathbb P \nb \cdot (u\otimes U_{high})\right\|_{L^{3,q}}
  \\&\quad+
  \int^{\rho t_1}_{0} 
  \left\|\left(e^{(t-\rho t_1)\Delta}-e^{(t_1-\rho t_1)\Delta}\right)e^{(\rho t_1-s)\Delta}\mathbb P \nb \cdot (u\otimes U_{low})\right\|_{L^{3,q}}
  \\&=
  I_{11} + I_{12}.
  \end{split}
  \end{equation}
For $I_{12}$, observe that for $s$, $\rho$ and $t_1$ fixed,  we have $(u \otimes U_{low})(s) \in L^{3,q}$ and so, by \eqref{Remark.oseen}, \[e^{(\rho t_1-s)\Delta}\mathbb P \nb \cdot[(u\otimes U_{low})(\tau)]  \in L^{3,q}.\] 
  Upon sending $t\to t_1$, 
this fact and
and 
\eqref{e.w08231} imply that
\begin{equation}
  \begin{split}
  \left\|\left(e^{(t-\rho t_1)\Delta}-e^{(t_1-\rho t_1)\Delta}\right)e^{(\rho t_1-s)\Delta}\mathbb P \nb \cdot (u\otimes U_{low})\right\|_{L^{3,q}}
  \rightarrow
  0.
  \end{split}
  \end{equation}
This amounts to pointwise convergence as $t\to t_1$ for $s\in (0,\rho t_1)$. 
We further have 
\begin{equation}
  \begin{split}
  &\left\|\left(e^{(t-\rho t_1)\Delta}-e^{(t_1-\rho t_1)\Delta}\right)e^{(\rho t_1-s)\Delta}\mathbb P \nb \cdot (u\otimes U_{low})(s)\right\|_{L^{3,q}}
  \\&\lesssim 
  \bigg( \frac {1} {(t-s)^{1/2}}   + \frac {1} {(t_1-s)^{1/2}}\bigg)s^{-1/2} \|   u\|_{L^\I(0,\I;L^{3,q})} \in L^1(0,\rho t_1).
  \end{split}
  \end{equation}
Applying the dominated convergence theorem now implies that $I_{12}\to 0$ as $t\to t_1$.
The argument for $I_{11}$ is identical once we observe that $L^{3,q}$ embeds continuously in $L^3$ and 
\[
\|  e^{(\rho t_1-s)\Delta}\mathbb P \nb \cdot [(u\otimes U_{high})(s)] \|_{L^3} \lesssim \frac 1 {(\rho t_1 -s)^{3/4}s^{1/2}} \| u\|_{\mathcal K_\I} \| U_{high}\|_{L^2}(s) <\I.  
\]
Hence, $e^{(\rho t_1-s)\Delta}\mathbb P \nb \cdot[(u\otimes U_{high})(\tau)]  \in L^{3,q}$.

  For $I_2$, we again use our decomposition of $U=U_{high}+U_{low}$, beginning with 
  \begin{equation}
  \begin{split}
  \|I_{2}\|_{L^{3,q}}
  &\leq
  \int_{\rho t_1}^{t} \|\nabla e^{(t_{1}-s)\Delta}\mathbb P  (u\otimes U_{high})\|_{L^{3,q}}\,ds
  +
  \int_{\rho t_1}^{t} \|\nabla e^{(t_{1}-s)\Delta}\mathbb P  (u\otimes U_{low})\|_{L^{3,q}}\,ds
  \\&=
  I_{21} + I_{22}.
  \end{split}
  \end{equation}
  For $I_{22}$, we use the fact that $\| U_{low}\|_{L^\I} (s)= \sqrt s^{-1}$ and get
  \begin{equation}
  \begin{split}
  I_{22}
  &\leq
  \int_{\rho t_1}^{t} 
  s^{-1/2}(t-s)^{-1/2}
  \|u\|_{L^{3,q}}
  \,ds
  \\&\leq
   \|u\|_{L^\I(0,\I;L^{3,q})}
   \int_{\rho t_1}^{t} 
  s^{-1/2}(t-s)^{-1/2}
  \,ds
  \\&\lesssim (\rho t_1)^{-1/2} (t-\rho t_1)^{1/2},
  \end{split}
  \end{equation}
  which can be made small  by taking $\rho$ close to $1$ and $t$ close to $t_1$.
For $I_{21}$ we have 
  \begin{equation}
  \begin{split}
  I_{21}
  &\leq 
  \int_{\rho t_1}^{t} \|\nabla e^{(t_{1}-s)\Delta}\mathbb P  (u\otimes U_{high})\|_{L^{3}}\,ds
  \\&\lesssim \int_{\rho t_1}^{t}
  (t-s)^{-3/4}
   \| u\|_{L^{10}} \| U_{high}\|_{L^{5/2}}\,ds
  \\&\lesssim
  \|u\|_{K_{10}}(\rho t_1)^{-7/20}\|U\|_{L^{3,\I}}\int_{\rho t_1}^{t} 
  (t-s)^{-3/4} s^{1/5} 
\,ds,
  \end{split}
  \end{equation}
 which can also be made small by taking $\rho$ close to $1$ and $t$ close to $t_1$.
The estimates for $I_3$ are essentially the same as those for $I_2$ and we omit them. The estimates for $B(u,u)(t)- B(u,u)(t_1)$ and $B(U,u)(t)- B(U,u)(t_1)$ are also similar and are omitted. 
Taken together, these bounds imply time-continuity in $L^{3,q}$ by first taking $\rho$ close to $1$ and then taking $t$ close to $t_1$.

\medskip 
 
{We now prove continuity at $t=0$. For this we use an inductive argument involving the Picard iterates $e_n$ in the proof of Proposition \ref{prop.fp}. We first observe that if $u_0\in L^{3,q}$ where $q<\I$, then $e_0= e^{t\Delta}u_0\to u_0$ in $L^{3,q}$ as $t\to 0^+$ and $t^{1/2-3/(2p)}\|e_0\|_{L^p}\to 0$ for $p\in (3,\I)$ as $t\to 0^+$---these properties fail when $q=\I$. 
Next, suppose these properties hold for $e_n$. We will show they also hold for $e_{n+1}$. Inspecting \eqref{BUlowSub}, \eqref{BUlowCrit}, \eqref{BUhighCrit} and \eqref{BUhighSub}, we can see that
\[
\|B(U,e_n) \chi_{(0,T]}\|_{X} \lesssim_{u_0,U} \|e_n \chi_{(0,T]}\|_{X}\to 0\text{ as } T\to 0^+,
\]
by assumption.
Recalling that $E$ in Proposition \ref{prop.fp} is what we are have  presently labeled $X$, we have that $e_n\to u$ in $X$ as $n\to \I$. This implies $e_n \chi_{(0,T]} \to u\chi_{(0,T]}$ in $X$. We therefore have that 
\[
\| (u-e_0) \chi_{(0,T]} \|_X \leq \|( u- e_n)\chi_{(0,T]} \|_X  + \| (e_n-e_0)\chi_{(0,T]}\|_{X}. 
\]
For any $\e>0$, we can make the right-hand side of the above small by first choosing $n$ large and then choosing $T$ small. Since $e_0\to u_0$ in $L^{3,q}$ as $t\to 0^+$, it follows from the above display that $u\to u_0$ in $L^{3,q}$.
}

\subsection{Modifications when $q=\I$.}
   {We now modify this argument for the case of $L^{3,\I}$ data. 
The fixed point argument is actually easier than in the case of   $L^{3,q<\I}$ data as we do not need to involve the Kato classes (although we could if we wanted to). Indeed, in \cite[Lemma 23]{Meyer}, Meyer shows that if $Z=L^{3,\I}$ and $E=L^\I(([0,\I);Z))\cap C([0,\I);Z)$, then  $B(\cdot,\cdot)$ is continuous from $E\times E$ into $E$. 
Since $U$ is in $Z=L^{3,\I}$ uniformly in time, there is no work to be done to conclude that 
\[
\| B(e,U) \|_E + \|B(U,e)\|_{E}\leq C_B \| e\|_E \|U\|_E,
\]
and so we can  apply Proposition \ref{prop.fp} to obtain the solution $u$ provided $\|U\|_E$ and $\|u_0\|_{L^{3,\I}}$ are small.

}
 	\end{proof}

\section{Asymptotic stability, or not}\label{sec:AS}

In this section we first prove Theorem \ref{thrm.asymptoticStability} and then prove Theorem \ref{thrm.excluded}.

Our proof of asymptotic stability, i.e.~Theorem \ref{thrm.asymptoticStability}, is based on re-framing the $L^3$-asymptotic stability problem in terms of the $L^2$-asymptotic stability theory of Karch et. al. Essentially, we will view a small element of $L^{3,q}$ as something that can be decomposed into an $L^{3,q}$ part---the tail--- and an $L^2$ part---the head. We can make the tail as small as we like and, applying Theorem \ref{thrm.globalWell} to it, we end up with a solution to the $U$-perturbed Navier-Stokes equations that is as small as we like. The $L^2$ part, evolving from the head, is now accounted for by the $L^2$-asymptotic stability of \cite{karch}.

\begin{proof}[Proof of Theorem \ref{thrm.asymptoticStability}] 
Let $\e>0$ be given. Suppose that $\|U\|\leq \e_1/{2}$ and $\|   u_0\|_{L^{3,q}}\leq \e_2 /2$. 
    Re-write $u_0$ as $\td u_0 + \bar u_0$ and assume $\| \bar u_0\|_{L^{3,q}} <\min\{ \e_1 / (2C), \e/(2C), \e_2/2\}$ while $\td u_0\in L^2$ and both are divergence free. This is done using the fact that $C_{c,\si}^\I$ is dense in $L^{3,q}_\si$ and, by definition, also in $\overline{C_{c,\si}^\I}^{L^{3,\I}}$. In particular, in our splitting $\td u_0\in C_{c,\si}^\I$.
    Let $\bar u$ solve \eqref{eq.PNS} with data $\bar u_0$ and perturbation term $U$; this comes from Theorem \ref{thrm.globalWell}. In particular we have 
    \[ \| \bar u \|_{L^\I(0,\I;L^{3,q})} <\min\{ \e_1/2, \e/2\}.
    \]
    Then, consider $\td u$ as  a solution to \eqref{eq.PNS} with initial data $\td u_0 = u_0 - \bar u_0$  with $U$ in \eqref{eq.PNS} replaced by $(U + \bar u)$. 
    Noting that \[\|U+\bar u\|_{L^\I(0,\I;L^{3,\I})}\leq \|U\|_{L^\I(0,\I;L^{3,\I})}+ \|\bar u\|_{L^\I(0,\I;L^{3,q})} \leq \e_1, \]
    and
    \[
    \| \td u_0\|_{L^{3,q}}\leq \| u_0\|_{L^{3,q}}+ \|\bar u_0\|_{L^{3,q}}\leq \e_2,
    \]
    we can still use Theorem \ref{thrm.globalWell} to solve for $\td u$. But $\td u_0$ is also chosen so that $\td u_0\in L^2$. Hence, the Karch et. al. theory \cite{karch} applies and generates a time-global Leray solution which must equal $\td u$ by weak-strong uniqueness, Theorem \ref{thrm.weakstronguniqueness}.  By Theorem \ref{thrm.karch}, we see that 
    \[
        \int_s^{s+1} \| \nb \td u\|_{L^2}^2 \,d\tau  \to 0
    \]
    as $s\to \I$, at least for almost every $s$. Noting that $\dot H^1$ embeds continuously into $L^6$, and since we can interpolate $L^3$ between $L^6$ and $L^2$, there must exist a time $t_0$ at which  $\| \td u\|_{L^{3,q}}(t_0)\leq \| \td u\|_{L^{3}}(t_0)\leq \min\{ \e_2/C,\e/(2C)\}$, where we have also used the continuity of the embedding $L^3\subset L^{3,q}$. We may now apply Theorem \ref{thrm.globalWell} a third time to conclude that $\sup_{t_0 <t <\I} \| \td u \|_{L^{3,q}}<\e/2$. 
    Consequently
    \[
    \sup_{t_0<t<\I} \|u \|_{L^{3,q}}<\e.
    \]
    Since $\e$ was arbitrary we conclude that 
    \[
    \lim_{t\to \I} \| u\|_{L^{3,q}} =0.
    \] 
\end{proof}

We now prove that Theorem \ref{thrm.asymptoticStability} is not true if $L^{3,q}$ is replaced by $L^{3,\I}$ with no further stipulations. Note that a solution $u$ is self-similar if it satisfies $u(x,t)=\la u(\la x ,\la^2 t)$ for every $\la>0$ and it is discretely self-similar if this possibly only holds for some $\la$. The initial data is self-similar or discretely self-similar if the preceding scaling relation holds with the time variable omitted. If we do not care that we are perturbing around a Landau solution, then the proof of Theorem \ref{thrm.asymptoticStability} is simple: We just take $u_0$ and $v_0$ to be self-similar, divergence free with small difference in $L^{3,\I}$.
The ensuing self-similar solutions $u$ and $v$ then confirm Theorem \ref{thrm.asymptoticStability} because the $L^{3,\I}$-norm of their difference is scaling invariant and hence does not decay as $t\to\I$. The same basic idea applies when the background flow is a Landau solution because it is also scaling invariant.    
Note that we cannot merely take the second solution to be another Landau solution  because, if two Landau solutions differ in $L^{3,\I}$ seminorm, then they have different forces and therefore the equation for their difference also has a forcing term. 

\begin{proof}[Proof of Theorem \ref{thrm.excluded}]
Let $U$ be  a Landau solution with small enough norm for Theorem \ref{thrm.globalWell} to apply.   We will understand this as a function of $x$ and $t$ where $U(x,t)= U(x)$. 

We recall a general fact about discretely self-similar vector fields:   $u_0\in L^{3,\I}\cap DSS$ if and only if $u_0\in L^3_\loc(\R^3\setminus \{0\})\cap DSS$ \cite{BT1}. To be more precise, in \cite{BT1} Tsai and Bradshaw showed that, if $u_0$ is $\la$-DSS then
\[
\int_{1\leq |x|\leq \la} |u_0|^3 \,dx\leq 3(\la-1)^2 \| u_0\|_{L^{3,\I}}^3,
\]
and 
\[
\|u_0\|_{L^{3,\I}}^3 \leq \frac {\la^3}{3(\la-1)} \int_{1\leq |x|\leq \la}|u_0|^3 \,dx,
\]
see \cite[(3.5) and (3.7)]{BT1}.
Let $\td u_0\in L^3$ satisfy $u_0 \in C_{c,\sigma}^\I (\{x:1\leq |x|\leq \la \})$ and $\| \td u_0\|_{L^3} = M>0$. Let $u_0$ be the $\la$-DSS extension of $\td u_0$ to $\R^3\setminus \{0\}$. Then, $u_0$ is divergence free in a distributional sense and $\|u_0\|_{L^{3,\I}} \sim_\la M$.   In this way we can construct discretely self-similar initial data of arbitrarily small size.

 Let $\e>0$ be given. Without loss of generality we take this less than $\e_2$ in Theorem \ref{thrm.globalWell} and  less than $ \frac{\| U\|_{L^{3,\I}}} 2$. Let $u_0$ be chosen so that $\| u_0\|_{L^{3,\I}} \leq \frac \e C$ where $C$ is as in Theorem \ref{thrm.globalWell}.
By Theorem \ref{thrm.globalWell}, there exists  a unique solution $u$  to \eqref{eq.PNS} with perturbation term $U$ and data $u_0$.  
Note that $u_0$ is discretely self-similar, as is $U$. Since $u$ is unique, by re-scaling we must have that $u$ is also discretely self-similar. 
Since we know that $u$ converges in a weak sense to $u_0$,   we cannot have $u=0$ in $L^{3,\I}$ for all positive times. In particular, there exists $t$ so that $u(\cdot,t)\neq 0$ in $L^{3,\I}$. But then by discretely self-similar scaling, $\| u(\cdot,\la^{2k} t)\|_{L^{3,\I}}= \|u(\cdot,t)\|_{L^{3,\I}} \neq 0$ in $L^{3,\I}$ for any $k\in \Z$. In particular,
\[
\limsup_{t\to \I} \|u\|_{L^{3,\I}}\geq \| u(\cdot, t)\|_{L^{3,\I}}.
\]
 
\end{proof}

\section*{Acknowledgements}

The research of Z.~Bradshaw was supported in part by the  NSF via grant DMS-2307097 and the  Simons Foundation via a TSM grant.  

W.~Wang was supported in part by the Simons Foundation via a TSM grant (No. 00007730), and he would like to thank Xukai Yan for discussions on Landau solutions.

	\bibliographystyle{siam}
	\bibliography{ref}

\bigskip\noindent 
Zachary Bradshaw, Department of Mathematics, University of Arkansas, Fayetteville, AR,  USA;
e-mail: \url{zb002@uark.edu}

\medskip\noindent 
Weinan Wang, Department of Mathematics, University of Oklahoma, Norman, OK, USA;
 e-mail: \url{ww@ou.edu}

\end{document}